\documentclass[a4paper, 10pt]{amsart}

\usepackage[T1]{fontenc}
\usepackage{amsmath}
\usepackage{amsfonts}
\usepackage{amssymb}
\usepackage{amsthm}
\usepackage[all]{xy}
\usepackage{enumerate}
\usepackage[symbol*]{footmisc}
\usepackage{graphicx}
\usepackage{parallel}
\usepackage{algorithm} 
\usepackage{algorithmic}
\usepackage{hyperref}

\DeclareMathOperator{\vol}{vol}
\DeclareMathOperator{\Tr}{Tr}

\DeclareMathOperator{\SL}{SL}
\DeclareMathOperator{\GL}{GL}

\DeclareMathOperator{\Grad}{Grad}

\DeclareMathOperator{\re}{Re}
\DeclareMathOperator{\md}{mod}
\DeclareMathOperator{\Aut}{Aut}

\title[Spherical Designs and Heights of Euclidean Lattices]{Spherical Designs and Heights of Euclidean Lattices}

\begin{document}
\setfnsymbol{chicago}
\setcounter{secnumdepth}{4}

\theoremstyle{plain}
\newtheorem{theor}{Theorem}
\newtheorem{prop}[theor]{Proposition}
\newtheorem{cor}[theor]{Corollary}
\newtheorem{lem}[theor]{Lemma}
\newtheorem{conj}[theor]{Conjecture}
\newtheorem{ass}[theor]{Assumption}

\theoremstyle{definition}
\newtheorem{defin}[theor]{Definition}
\newtheorem*{thm}{Theorem} 
\newtheorem{notation}[theor]{Notation}

\theoremstyle{remark}
\newtheorem{rem}[theor]{Remark}
\newtheorem*{rem*}{Remark}
\newtheorem*{cons*}{Consequence}
\newtheorem{ex}[theor]{Example}
\newtheorem*{conc*}{Conclusion}

\newcommand{\mbb}[1]{\mathbb{#1}}
\newcommand{\mbf}[1]{\mathbf{#1}}
\newcommand{\mc}[1]{\mathcal{#1}}
\newcommand{\Q}{\mathbf{Q}}
\newcommand{\R}{\mathbf{R}}
\newcommand{\Z}{\mathbf{Z}}
\newcommand{\I}{\mathbf{I}}
\newcommand{\C}{\mathcal{C}}
\newcommand{\D}{\mathcal{D}}
\newcommand{\OK}{\mathcal{O}_K}
\newcommand{\pd}{\partial}

\author{Renaud Coulangeon}
\address{Univ. Bordeaux, IMB, UMR 5251, F-33400 Talence, France.\\
CNRS, IMB, UMR 5251, F-33400 Talence, France.}
\email{renaud.coulangeon@u-bordeaux.fr}
\author{Giovanni Lazzarini}
\email{Giovanni.Lazzarini@u-bordeaux.fr}

\begin{abstract}
We study the connection between the theory of spherical designs and the question of extrema of the height function of lattices. More precisely, we show that a full-rank $n$-dimensional Euclidean lattice $\Lambda$, all layers of which hold a spherical $2$-design, realises a stationary point for the height $h(\Lambda)$, which is defined as the first derivative at the point $0$ of the spectral zeta function of the associated flat torus $\zeta(\mathbf{R}^n /\Lambda)$. Moreover, in order to find out the lattices for which this $2$-design property holds, a strategy is described which makes use of theta functions with spherical coefficients, viewed as elements of some space of modular forms. Explicit computations in dimension $n\leq 7$, performed with Pari/GP and Magma, are reported.        
\end{abstract}

\maketitle

\section{Introduction}
The aim of this article is to investigate (local) extremality properties of the height on the set of lattices of covolume $1$, and to describe its stationary points in terms  of spherical designs.

The \emph{height} of a Euclidean lattice $\Lambda$ is defined as the derivative at the point $s=0$ of the spectral zeta function of the flat torus associated to $\Lambda$. More generally, if $(X, g) $ is a compact connected Riemannian manifold without boundary, the spectrum of the associated Laplace operator is a discrete sequence of real numbers 
$0 = \lambda_0 < \lambda_1 \leq \lambda_2 \leq \cdots $, converging to infinity, and the \emph{spectral zeta function} is then defined, for $\re (s) >>0$, by the series 
\begin{equation}
\zeta_{(X, g)}(s) = \sum_{i=1}^{\infty} \dfrac{1}{\lambda_i^s}. 
\end{equation}    
This expression is well-known to admit a meromorphic continuation to $\mathbf C$, holomorphic at $0$ (see \cite{BGM}, \cite{GHL}). This allows to  give meaning to the determinant of the Laplacian, \textit{via} the standard zeta regularisation process  
\begin{equation}
\det \Delta_{g}:= \exp (-\zeta'_{(X, g)}(0)),
\end{equation}   
which indeed can be interpreted as the formal infinite product of the $\lambda_i$'s. Then the height of $(X,g)$ is simply $h(X)= \zeta'_{(X, g)}(0)$.
A natural problem arises at this point, namely to maximise $\det \Delta_{g}$ on a given manifold $X$ as the metric $g$ varies (with fixed volume) and characterise the optimal metrics (see \cite{OPS}). 

When $X$ is a flat torus $\mathbf{R}^n / \Lambda$ associated to a full-rank lattice $\Lambda$ in the Euclidean space $\mathbf{R}^n$, then this question seems to be more tractable, though far from trivial. Indeed, for such a lattice $\Lambda$ one can construct the \emph{Epstein zeta function}
\begin{equation}
Z(\Lambda, s) := \sum_{0\neq x\in \Lambda} \dfrac{1}{\lVert x \rVert ^{2s}},\quad s\in \mathbf{C} \text{ with } \re(s)>n/2, 
\end{equation}    
which is well-known (since Epstein \cite{Ep03}) to admit a meromorphic continuation to the complex plane with a simple pole at $ s=n/2 $.
Besides, the eigenvalues of the Laplacian on  $\mathbf{R}^n / \Lambda$ are exactly $ 4\pi^2 \lVert x\rVert^2 $, for $ x $ in the standard dual 
\begin{equation}
 \Lambda^*= \lbrace y\in \R^n \, |\, (x \cdot y)\in \Z \text{ for all } x\in \Lambda\rbrace 
 \end{equation}
(see \cite{Chiu} and \cite{Sar06}), whence we have the following identity between the two zeta functions
 \begin{equation}
\zeta_{\mathbf{R}^n / \Lambda}(s) = (2\pi)^{-2s}Z(\Lambda^*, s),
\end{equation}
which gives immediately the one for the height of $ \Lambda $:
\begin{equation}
h(\Lambda)= \zeta_{\mathbf{R}^n / \Lambda}'(0)= Z'(\Lambda^*, 0) + 2\log(2\pi).
\end{equation}  
Since for $c >0$ we have that $Z(c\Lambda, s) = c^{-2s}Z(\Lambda, s)$, the question of the minima of the height makes sense only if we restrict ourselves to lattices of fixed covolume, usually $1$. From now on, let $\mathcal{L}_n^\circ$ denote the set of lattices of determinant $1$, hence of covolume $1$: then our task of finding a minimum of the height function restricted to flat tori amounts to minimise $Z'(\Lambda^*, 0)$ for $\Lambda \in\mathcal{L}_n^\circ $.

Chiu has shown in \cite{Chiu} that a global minimum exists in every dimension $n$, but the actual minimum is known only in dimension $2$ and $3$, where it is achieved respectively by the hexagonal lattice (Osgood, Phillips and Sarnak, \cite{OPS}) and the face-centred cubic lattice (Sarnak and Str\"ombergsson, \cite{SS}). Moreover, in the same paper by Sarnak and Str\"ombergsson it is proved that (the tori associated to) the $D_4$ lattice (rescaled so as to have determinant $1$), the $E_8$ lattice and the Leech lattice $\Lambda_{24}$ achieve a strict \emph{local} minimum of $h$ in dimension $4$, $8$ and $24$ respectively [the standard definition and properties of the lattices mentioned in this paper can be found in \cite{CS} or in \cite{Mar}]. Their proof goes through an inspection of the automorphism group of these lattices and makes use of quite sophisticated group theoretic tools.

The first author in his paper \cite{Cou1} has suggested a more general approach to the problem, which involves the notion of spherical design. 
A \emph{spherical $t$-design} is a finite set of points on a sphere which is well distributed from the point of view of numerical integration: more precisely, a finite subset $X$ on a sphere $\mathbf{S} \subset \mathbf R^n$ centred at $O$ is a spherical $t$-design if for every homogeneous polynomial $f$ of degree $\leq t$, 
\begin{equation}\label{eq3}
\int _{\mathbf{S}} f(x)\, dx = \dfrac{1}{\lvert X\rvert} \sum_{x\in X} f(x),
\end{equation}
where $dx$ is any $O(n)$-invariant measure on $\mathbf{S}$ (for example the surface measure on $\mathbf{S}$ induced by the Lebesgue measure on $\mathbf R^n$).  For us, the standard way of constructing such objects is to consider what are called the \emph{layers} or \emph{shells} of a lattice, that is, the sets of vectors of a given length.


The concept of a spherical design is due to Delsarte, Goethals and Seidel in their paper \cite{DGS}, but the connection between spherical designs held by the layers of a lattice and its Epstein zeta function had already been observed. 
In particular, Delone and Ryshkov had characterised in \cite{DR67} the lattices in  $\mathcal{L}_n^\circ $ that are $Z$-\emph{extreme} at any large enough $s$ (that is, that achieve a strict local minimum of the function $\Lambda \mapsto Z(\Lambda, s)$), and one of the conditions, stated in a different terminology, is that all layers of the lattice hold a $2$-design. 
Other results strengthen the idea that if a lattice $\Lambda$ holds spherical designs of strength $2$ or $4$ on some of its layers, then it should have good properties in terms of density of the associated sphere packing and of the zeta function. For example, Venkov has proved in \cite{V01} that a lattice whose first layer (i.e. its \emph{minimal vectors}) is a $4$-design achieves a local maximum of the density function. 

Coming to the height, the first author's result is as follows:
\begin{thm}[Coulangeon, 2006]
Let $\Lambda\in\mathcal{L}_n^\circ$ be such that all its layers hold  a $4$-design; then $\Lambda$ is $Z$-extreme at $s$ for any $s>n/2$, and moreover the torus associated with its dual $\Lambda^*$ achieves a strict local minimum of the height on the set of $n$-dimensional flat tori of covolume $1$.
\end{thm} 
This is the first result which links the height of a lattice to spherical designs, and it applies to $D_4$, $E_8$ and $\Lambda_{24}$, as well as to a large amount of \emph{extremal modular lattices}, for which it is possible to show that all their layers hold a $4$-design (see \cite{BV}). 
However, the condition that all the layers of a given lattice  hold a $4$-design is too strong with respect to the quest of extrema of the height function: indeed, lattices whose minimal vectors form a $4$-design, the so-called \emph{strongly perfect lattices}, have been classified in dimension up to $12$ (see \cite{NV00, NV05, Q95}), and for $n=3$, $5$ and $9$ it has been proved that such a lattice does not exist, hence a fortiori a lattice with $4$-designs on every shell.
This clearly contrasts with the fact, recalled above, that the height attains a global minimum in every dimension $n$, therefore the right characterisation of the local minima of the height function is still to be found. In this direction, a contribution is our main theorem, which we state as follows:
\begin{theor}\label{th1}
Let a lattice $\Lambda\in\mathcal{L}_n^\circ$ be such that all its layers hold  a spherical $2$-design; then $\Lambda$ is a stationary point for the function $\Lambda \mapsto h(\Lambda)$.
\end{theor} 
Here, the term ``stationary (or critical) point'' is intended in the sense of differentiable manifolds, and its precise meaning will be explained later (see Section \ref{sec1}). We call such lattices \emph{fully critical} with respect to the height; the term is in analogy with the notion of fully eutactic lattice defined in \cite{Gruber}.    

Theorem \ref{th1} suggests naturally the following question: in a given dimension,  which are the lattices which hold a spherical $2$-design on every layer? It is a well-known result that such lattices exist in every dimension (see Section \ref{sec2}), and there are already several examples of families of lattices sharing this property (see for example \cite{Cou1} for Barnes-Wall lattices, and \cite{BV} for some extremal modular lattices). In the second part of this paper, we focus on a test to decide whether a given lattice satisfies the condition of Theorem \ref{th1} or not. This test uses theta functions with spherical coefficients: they are weighted theta series attached to any lattice $\Lambda$, which under certain conditions are modular forms for some subgroup of $\SL_2(\Z)$. There is a connection between the Fourier coefficients of these theta series and the $t$-design property of the layers of $\Lambda$, which roughly speaking allows us to conclude that if all the layers of $\Lambda$ up to a certain radius are $2$-designs, then actually all the layers are. 
Thanks to this, we were able to classify all lattices with the $2$-design property in dimension $2$ to $6$: the computations have been performed with the computer algebra systems Magma and Pari/GP.

The paper is organised as follows: is Section \ref{sec3} we collect some preliminary material about lattices and quadratic forms, spherical designs and theta series with spherical coeficients. Section \ref{sec1} contains the proof of Theorem \ref{th1}, and in Section \ref{sec2} we discuss the problem of finding lattices which have $2$-designs on every layer, distinguishing the cases of even and odd $n$. Section \ref{sec4} contains some remarks, open questions and a comparison with other known results. The tables of fully critical lattices in dimension $2$ to $6$ and some examples in dimension $7$, as well as an explicit example of the computations described in Section \ref{sec2} are given in the appendices. 

\subsection*{Notation}
Throughout the paper, $(x\cdot y)$ will denote the usual inner product of two elements $x$ and $y$ of the vector space $\R^n$, and $\lVert x\rVert$ the associated norm. The vectors in $\R^n$ are considered as column vectors. The transpose of a matrix $A$ is denoted by $A^t$, and if $A$ is a squared matrix of size $n$, and $x$ a vector of $\R^n$, the notation $A[x]$ stands for the product $x^t A x$.

The set of real symmetric matrices of size $n$ is denoted by $\mathcal{S}_n(\R)$. It is a vector space endowed with a canonical inner product given by the trace:
\begin{equation}
\langle A , B\rangle := \Tr AB, \quad A, B\in \mathcal{S}_n(\R).
\end{equation}   

Finally, in this paper the expression "$\Lambda$ has (or holds) a (spherical) $t$-design on a layer" means "the whole set of points given by that layer is a $t$-design".  

\section{Basics}\label{sec3}
\subsection{Lattices and quadratic forms}
A Euclidean lattice in $\R^n$ is a subset $\Lambda \subset \R^n$, equipped with its standard inner product, with the property that there exists a basis $\lbrace v_1, \dots, v_n\rbrace$ of $\R^n$ such that 
$$ \Lambda = \Z v_1 \oplus \cdots \oplus \Z v_n, $$
i.e. $\Lambda$ consists of all integral linear combinations of the vectors $v_1, \dots, v_n$. Remark that this definition is purposedly slightly restrictive and corresponds to what is often called a \emph{full-rank} lattice in the literature.
The dual lattice of a Euclidean lattice $\Lambda$ is
\begin{equation}
 \Lambda^*= \lbrace y\in \R^n \, |\, (x \cdot y)\in \Z \text{ for all } x\in \Lambda\rbrace .
 \end{equation}
 Let $A$ be the matrix which has the $v_i$'s as columns, and consider the matrix 
\begin{equation}
Q= A^t A =(v_i \cdot v_j)_{1\leq i, j\leq n};
\end{equation}  
Then $Q$ is a \emph{Gram matrix} for $\Lambda$, and the \emph{covolume} of $\Lambda$ is defined as
\begin{equation}
\vol (\Lambda) = \lvert \det A\rvert = \sqrt{\det Q}.
\end{equation}
We consider $Q$ as the matrix of a positive definite quadratic form: indeed if $x\in \Lambda$, say $x= Am$ for $m\in \Z^n$, then $\lVert x\rVert ^2 = (Am)\cdot (Am) = Q[m]$.
If $\mathcal{L}_n$ (resp. $\mathcal{L}_n^\circ$) is the set of Euclidean lattices (resp. of covolume $1$) in $\R^n$, and $\mathcal{P}_n$ (resp. $\mathcal{P}_n^\circ$) the cone of positive definite quadratic forms (resp. of determinant $1$) in $n$ variables, then there is a well-known ``dictionary'' between Euclidean lattices and positive quadratic forms, which allows to formulate every definition and statement in either of these languages, depending of which is more convenient for clarity's sake. In particular, for the proof of Theorem \ref{th1} it is easier to work with quadratic forms. Notice that if $Q$ is a Gram matrix for $\Lambda$, then $Q^{-1}$ is a Gram matrix for $\Lambda^*$. 
We identify a quadratic form in $\mathcal{P}_n$ with its matrix in the canonical basis of $\R^n$: the map $A \mapsto A^t A$ induces a bijection from $\mathrm{O}_n(\R) \backslash \GL_n(\R)$ onto $\mathcal{P}_n$, hence we can identify these two sets; similarly, the map associating the matrix $A$ to the lattice $\Lambda = A\Z^n$ induces a bijection between $\GL_n(\R) / \GL_n(\Z)$ and $\mathcal{L}_n$. The situation is summarised by the following diagram, which provides the important one-to-one correspondence between classes of lattices modulo isometry and classes of positive definite quadratic forms modulo $\GL_n(\Z)$-equivalence   :
\begin{equation} 
\xymatrix{
 & \GL_n(\R) \ar[dl] \ar[dr] & \\
\mathcal{P}_n = \mathrm{O}_n(\R)\backslash \GL_n(\R) \ar[dr] &  & \GL_n(\R)/ \GL_n(\Z) =\mathcal{L}_n \ar[dl] \\
 & \mathcal{P}_n/ \GL_n(\Z) = \mathrm{O}_n(\R)\backslash \mathcal{L}_n &  
}
\end{equation}
A lattice is called \emph{integral} if $(x\cdot y) \in \Z$ for every $x,y \in \Lambda$; an integral lattice is called \emph{even} if $(x\cdot x) \in 2\Z$ for all $x \in \Lambda$, and \emph{odd} otherwise. Let $\Lambda = A\Z^n$ be a lattice in $\R^n$, and let $Q=A^t A $ the corresponding quadratic form. We define the sequence of increasing squared lengths of non-zero vectors in $\Lambda$,
\begin{equation}
m_1(\Lambda)< m_2(\Lambda) < \cdots
\end{equation}
 and the $k$-th \emph{layer} of $\Lambda$ is thus defined as
 \begin{equation}
M_k(\Lambda) := \lbrace x \in \Lambda \, |\, (x\cdot x) = m_k(\Lambda) \rbrace
\end{equation}  
Similarly, one defines the increasing sequence $m_1(Q)< m_2(Q)<\cdots$ of the non-zero values attained by $Q$, and the associated layers 
\begin{equation}
M_k(Q) := \lbrace x \in \Z^n \, |\, Q[x] = m_k(Q) \rbrace.
\end{equation}
\subsection{Spherical designs} 
The reference for the theory of spherical designs, viewed in connection with Euclidean lattices, is the work by Venkov \cite{V01}. Here we just collect some properties which are needed in the sequel. As a definition of spherical design, we take the one given in the introduction of this paper, but there are many other characterisations of a spherical design, some of which are listed in the following Proposition \ref{prop2}.

Identify the polynomial ring $\R[X_1, \dots, X_n]$ with the set $\mathcal{F}$ of polynomial functions on $\R^n$ via the canonical basis $\mathcal{E}$ of $\R^n$, and denote by $\mathcal{F}_m$ the homogeneous part of degree $m$ of $\mathcal{F}$. The Laplace operator $\Delta = \partial^2/\partial x_1^2 + \cdots \partial^2/\partial x_n^2$ is invariant under $\mathrm{O}_n(\R)$ and it maps $\mathcal{F}_m$ onto $\mathcal{F}_{m-2}$. Its kernel consists of the homogeneous harmonic polynomials of degree $m$,
\begin{equation}
\ker \Delta = \mathrm{Harm}_m.
\end{equation}   
\begin{prop}[see \cite{V01}, Th. 3.2] \label{prop2}
Suppose $n\geq 2$ and  let $X$ be a finite subset of $\mathbf{S}(r)=\left\lbrace x \in \mathbf R^n , \Vert x \Vert=r \right\rbrace $.  Assume that $X$ is symmetric about the origin $0$, i.e. $X=-X$. Then, for any even positive integer $t$, the following conditions are equivalent:
\begin{enumerate}[(i)]
\item $X$ is a spherical $t$-design, i.e. 
\begin{equation}
\int _{\mathbf{S}} f(x)\, dx = \dfrac{1}{\lvert X\rvert} \sum_{x\in X} f(x),
\end{equation}
 for every homogeneous polynomial $f$ of degree $\leq t$;
\item for every non-constant harmonic polynomial $P(x)$ of degree $\leq t$,
\begin{equation}
\sum_{x\in X} P(x) = 0.
\end{equation}
\item there exists a constant $c$ such that for all $\alpha \in \R^n$, 
\begin{equation}
\sum_{x\in X} (x\cdot \alpha)^{t} = c r^{t/2}(\alpha\cdot \alpha)^{t/2}
\end{equation} 
\end{enumerate}
\end{prop} 
The proof is given in \cite{V01}, with the only difference that the condition for all $\alpha \in \R^n$ 
\begin{equation}
\sum_{x\in X} (x\cdot \alpha)^{t-1}=0, 
\end{equation}
which should appear in $(iii)$, is automatically satisfied here because $X$ is symmetric around the origin. The value of the constant $c$ is 
\begin{equation}
c = c_t = \dfrac{1 \cdot 3 \cdot 5\cdots (t-1)}{n(n+2) \cdots (n+t-2)}\lvert X\rvert.
\end{equation}  
In particular, if we want to prove that a layer of a lattice $\Lambda$ is a $2$-design, condition $(iii)$ of Proposition \ref{prop2} specialises as follows:
\begin{cor}\label{cor1}
A non-empty layer $M_k(\Lambda)$ of the lattice $\Lambda = A\Z^n$ is a $2$-design if and only if one of the following equivalent conditions holds:
\begin{enumerate}[(i)]
\item for all $\alpha \in \R^n,$
\begin{equation}
\sum_{x\in M_k(\Lambda)} (x\cdot \alpha)^2 = \dfrac{1}{n} m_k(\Lambda) (\alpha \cdot \alpha) \lvert M_k(\Lambda)\rvert; \end{equation}
\item \begin{equation}
\sum_{y\in M_k(\Lambda)} y y^t = \dfrac{1}{n} m_k(\Lambda) \lvert M_k(\Lambda)\rvert \mathbf{I}_n;
\end{equation}
\item \begin{equation}
\sum_{m\in M_k(Q)} m m^t = \dfrac{1}{n} m_k(Q) \lvert M_k(Q)\rvert \,Q^{-1}.
\end{equation}
\end{enumerate}
\end{cor}

An important characterisation of $t$-designs, which enables to implement a computational test on a set to decide whether it is a $t$-design, is given by the following proposition:
\begin{prop}[see \cite{V01}, Th. 8.1]\label{prop3}
Let $X$ be symmetric around the origin and of squared norm $m$, and let $t$ be even. Then
\begin{equation}
\sum_{x,y\in X} (x\cdot y)^t \leq \dfrac{1 \cdot 3 \cdot 5 \cdots (t-1)}{n(n+2) \cdots (n+t-2)} m^t \lvert X\rvert^2
\end{equation}
and the equality holds if and only if $X$ is a $t$-design.
\end{prop} 
In particular, a layer $M_k(\Lambda)$ of a lattice $\Lambda$ is a $2$-design if and only if
\begin{equation}
\sum_{x,y \in M_k(\Lambda)}(x\cdot y)^2 = \dfrac{1}{n} m_k(\Lambda)^2 \lvert M_k(\Lambda)\rvert^2.  
\end{equation} 
A lattice is called \emph{strongly eutactic} its minimal vectors $M_1(\Lambda)$ are a $2$-design, and \emph{strongly perfect} if they are a $4$-design.
\subsection{Modular forms and Theta series}
Let $\Lambda$ be a lattice in dimension $n$, and $P$ a harmonic polynomial of degree $r$; let $\mathbf{H}$ denote the set of complex numbers with positive imaginary part, then for $\tau \in \mathbf{H}$ the weighted theta series is defined as
\begin{equation}
\theta_{\Lambda, P}(\tau) = \sum_{x\in \Lambda} P(x) e^{\pi i \tau (x\cdot x)}.
\end{equation} 
The function $\theta_{\Lambda, P}$ is holomorphic in $\mathbf{H}$. If $P=1$, we get the classical theta series
\begin{equation}
\theta_{\Lambda}(\tau) = 1+ \sum_{k\geq 1} \lvert M_k(\Lambda)\rvert e^{\pi i \tau m_k(\Lambda)} 
\end{equation} 
If $\Lambda$ is integral and even, and $P$ is non-constant, we can rewrite $\theta_{\Lambda, P} $ as
\begin{equation}
\theta_{\Lambda, P} = \sum_{k=1}^\infty \left(\sum_{x\in M_k(\Lambda)} P(x) \right) q^{m_k(\Lambda)/2}
\end{equation}
where $q= e^{2 \pi i \tau}$. From this we understand why the weighted theta series are useful to study the design structure of the layers of $\Lambda$:
\begin{prop}\label{prop6}
Let $\Lambda$ be even and integral in $\R^n$. For $k >0$, the layer $M_k(\Lambda)$ is a spherical $t$-design if and only if the coefficient of $q^{m_k(\Lambda)/2}$ in the Fourier development of $\theta_{\Lambda, P}$ is zero for every harmonic polynomial of even degree $r\leq t$.
\begin{cor}\label{cor2}
A lattice $\Lambda$ as above has $t$-designs on every layer if and only if 
\begin{equation}
\theta_{\Lambda, P}\equiv 0
\end{equation}
for every harmonic $P$ of even degree $r\leq t$.
\end{cor}
\end{prop}
The theta series of $\Lambda $ weighted by $P$ is a modular form for some subgroup of $\SL_2(\R)$. The general reference for modular forms is the book by Miyake \cite{M}; however the definition of a modular form and the result that we need (Proposition \ref{prop5}) are contained in Chapter 2 and 3 of the book by Ebeling \cite{Eb}.
If $\Gamma $ is a subgroup of $\mathrm{SL}_2(\Z)$ and $\chi $ is a character of $\Gamma$, we denote by $\mathcal{M}_k(\Gamma, \chi)$ the space of modular forms with respect to $\Gamma$, of weight $k$ and character $\chi$.  We recall the following well-known subgroups of finite index of $\SL_2(\Z)$ (with  $N\geq 1 $ an integer):
\begin{eqnarray}
\Gamma_0(N)&:=& \Big\lbrace \begin{pmatrix} a & b \\ c & d\end{pmatrix} \in \SL_2(\Z) \,|\,N |c \Big\rbrace \\
\Gamma_1(N)&:=& \Big\lbrace \begin{pmatrix} a & b \\ c & d\end{pmatrix} \in \SL_2(\Z) \,|\, a\equiv d\equiv 1\md N, \ c\equiv 0 \md N \Big\rbrace \\
\Gamma(N)&:=& \Big\lbrace \begin{pmatrix} a & b \\ c & d\end{pmatrix} \in \SL_2(\Z) \,|\, \begin{pmatrix} a & b \\ c & d\end{pmatrix} \equiv \begin{pmatrix} 1 & 0 \\ 0 & 1\end{pmatrix}\md N \Big\rbrace 
\end{eqnarray}

Now let $\Lambda$ be an even integral lattice, of dimension $n$ with even $n$.   
\begin{defin}
The \emph{level} $\ell$ of $\Lambda$ is the smallest integer such that the lattice $\sqrt{\ell} \Lambda^*$ is also even.
\end{defin}
 Equivalently, consider the quotient $D = \Lambda^* / \Lambda$, endowed with the quadratic form $x \mapsto d(x) = \frac{1}{2}(x,x)$, with image in $\Q / \Z$. Then $\ell$ is the annihilator of $D$, which is also the smallest integer such that $\ell d$ is the zero form. 
 
 The following result, regarding modularity properties of the above defined weighted theta series, is classical, but we nevertheless include a sketch of proof, for the sake of self-containedness.
 \begin{prop}\label{prop5}
Let $\Lambda\subset \R^n$ ($n$ even) be an even integral lattice of level $\ell$, and $P$ a harmonic  polynomial of degree $r$, then $\theta_{\Lambda, P}$ is a modular form of weight $k= n/2 + r$ for the group $\Gamma_1(\ell)$, that is,
\begin{equation}
\theta_{\Lambda, P} \in \mathcal{M}_{n/2+ \deg P}(\Gamma_1(\ell)).
\end{equation} 
\end{prop}
\begin{proof}
For the property of modularity, it is a particular case of \cite{Eb}, Corollary 3.1 and Theorem 3.2, because the character which appears there is trivial on $\Gamma_1(\ell)$, as it is shown in the proof of the same Theorem 3.2. The fact that $\theta_{\Lambda, P}$ is holomorphic at every cusp of $\Gamma_1(\ell)$ is proved in \cite{M}, Corollary 4.9.4, and this completes the proof that  $\theta_{\Lambda, P}$ is a modular form for $\Gamma_1(\ell)$.
\end{proof}
The space $\mathcal{M}_{n/2 + \deg P}(\Gamma_1(\ell))$ is of finite dimension over $\Q$, and a basis can be computed algorithmically by a computer algebra system such as Magma, as we will do in Section \ref{sec2}.


\section{Proof of Theorem \ref{th1}}\label{sec1}
In order to study the local behaviour of the height around a given element $Q_0 \in \mathcal{P}_n^\circ$, we use an expression of it via an integral representation. This is an old identity, essentially due to Riemann, see \cite{Ter1}. 
\begin{lem}[Riemann]\label{lem1}
The Epstein zeta function $Z(Q, s)$ can be expressed as
\begin{multline}
\pi^{-s}\Gamma(s)Z(Q, s) = \dfrac{\lvert Q\rvert ^{-1/2}}{s-n/2} -\dfrac{1}{s}  \\
+ \sum_{m\in\Z^n\setminus\lbrace 0\rbrace} \Big(\int_{1}^\infty e^{-\pi Q[m]t}t^s \, \dfrac{dt}{t} + \int_1^\infty e^{-\pi Q^{-1}[m]t}t^{n/2-s}\, \dfrac{dt}{t}\Big) 
\end{multline} 
\end{lem}
By Lemma \ref{lem1}, we have
\begin{multline}\label{eq5}
Z (Q^{-1},s)=\dfrac{\pi^s}{\Gamma(s)} \Big(\dfrac{1}{s-n/2} -\dfrac{1}{s}\Big)+ \\
+\dfrac{\pi^s}{\Gamma(s)} \underbrace{ \sum_{m\in \mathbf{Z}^n\setminus \lbrace 0\rbrace} \Big(\int_1^\infty e^{-\pi Q^{-1}[m]t} t^s\, \dfrac{dt}{t} + \int_1^\infty e^{-\pi Q[m]t} t^{n/2-s}\, \dfrac{dt}{t}  \Big)}_{F_Q(s)}.
\end{multline}
We differentiate with respect to $s$ and evaluate in $s=0$ to get the height, as Chiu does in \cite{Chiu}. The first term in \eqref{eq5} is analytic at $s=0$ and independent of $Q$, so it contributes a constant $C$. Then
\begin{equation}
\dfrac{d}{ds}Z(Q^{-1}, s)\Big|_{s=0} = C + \Big[ \dfrac{\pi^s}{\Gamma(s)}F'_Q(s)+ \dfrac{\pi^s \log \pi}{\Gamma(s)}F_Q(s) +\pi^s \, \dfrac{d}{ds}\Big(\dfrac{1}{\Gamma(s)} \Big) F_Q(s) \Big]_{s=0}.
\end{equation}

As $F_Q(0)$ and $F'_Q(0)$ clearly converge and $1/\Gamma(0) = 0$, the second and third term vanish. In the last term, we have
$$  \Big[\pi^s \, \dfrac{d}{ds}\Big(\dfrac{1}{\Gamma(s)} \Big)\Big]_{s=0} =1, $$
because the reciprocal Gamma function is an entire function with Taylor series
$$ 1/\Gamma(s) = s + \gamma s^2 + O(s^3) $$
(see \cite{Bourguet}, or \cite{wrench}). Thus we are left with
\begin{equation}
\dfrac{d}{ds}Z(Q^{-1}, s)\Big|_{s=0} = C + F_Q(0).  
\end{equation}
We view $\mathcal{P}_n^\circ$ as a differentiable submanifold of $\mathcal{S}_n(\R)$ (see \cite{Bav}) . The tangent space $\mathcal{T}_Q$ at a point $Q$ identifies with the set $\lbrace H\in \mathcal{S}_n(\R) \,| \langle Q^{-1}, H\rangle =0 \rbrace$, and the exponential map $H\mapsto e_Q(H)= Q\exp(Q^{-1} H)$ induces a local diffeomorphism from $\mathcal{T}_Q$ onto $\mathcal{P}_n^\circ$. 
What we have to do is then to study the local behaviour of $F_Q(0)$, around $Q=Q_0$, via $Q= e_{Q_0(H)}$, and to do that we compute the Taylor expansion at the first order of the map $H \mapsto F(e_{Q_0}(H),0)$, for $H \in \mathcal{T}_{Q_0}$, that is 
\begin{equation}
F(e_{Q_0(H)}, 0)= F(Q_0,0)+ \langle \Grad F(Q_0, 0), H\rangle  + o(\lVert H\rVert).
\end{equation} 
The computation goes as follows:
the Taylor expansions of the exponential map
\begin{equation}\begin{split}
Q= e_{Q_0(H)}=Q_0\exp(Q_0^{-1} H) &= Q_0(\I+ Q_0^{-1} H + 1/2 (Q_0^{-1}H)^2 +o(\lVert H\rVert^2) )\\
&= Q_0 + H + 1/2 H Q_0^{-1} H +o(\lVert H\rVert^2)
\end{split}\end{equation}
and
\begin{equation}\begin{split}
Q^{-1}=\exp(-Q_0^{-1} H)Q_0^{-1} &= (\I - Q_0^{-1} H + 1/2 Q_0^{-1} H Q_0^{-1} H) Q_0^{-1} + o(\lVert H\rVert^2) \\
&=  Q_0^{-1} - Q_0^{-1} H Q_0^{-1} + 1/2 Q_0^{-1}H Q_0^{-1}H Q_0^{-1} + o(\lVert H\rVert^2)
\end{split}\end{equation}
put into the expression for $F(e_{Q_0(H)}, 0)$ give
\begin{equation}\begin{split}
F(e_{Q_0(H)}, 0) &= \sum_{m\in \mathbf{Z}^n\setminus \lbrace 0\rbrace} \int_1^\infty e^{-\pi (Q_0^{-1} - Q_0^{-1} H Q_0^{-1} + 1/2 Q_0^{-1}H Q_0^{-1}H Q_0^{-1} + o(\lVert H\rVert^2)) [m]t}\, \dfrac{dt}{t}\\
 &+ \sum_{m\in \mathbf{Z}^n\setminus \lbrace 0\rbrace}\int_1^\infty e^{-\pi (Q_0 + H + 1/2 H Q_0^{-1} H +o(\lVert H\rVert^2))[m]t} t^{n/2}\, \dfrac{dt}{t}.  
 \end{split}\end{equation}
Since 
\begin{equation}\begin{split}
e^{-\pi (H + 1/2 H Q_0^{-1}H + o(\lVert H\rVert^2))[m]t} &= 1 - \pi (H + 1/2 H Q_0^{-1} H + o(\lVert H\rVert^2))[m]t + o(\lVert H\rVert^2) \\ &= 1 - \pi H[m] t + o(\lVert H\rVert) 
\end{split}\end{equation}
the second sum contributes with
\begin{multline}
\sum_{m\in \mathbf{Z}^n\setminus \lbrace 0\rbrace} \int_1^\infty e^{-\pi Q_0[m]t} t^{n/2} \, (1 - \pi H[m] t + o(\lVert H\rVert))\dfrac{dt}{t}\\
= \sum_{m\in \mathbf{Z}^n\setminus \lbrace 0\rbrace} \int_1^\infty e^{-\pi Q_0[m] t } t^{n/2} \, \dfrac{dt}{t} - \sum_{m\in \mathbf{Z}^n\setminus \lbrace 0\rbrace} \int_1^\infty e^{-\pi Q_0[m]t} t^{n/2} \pi H[m] \, dt + o(\lVert H\rVert).
\end{multline}
For the first sum, since
\begin{equation}
e^{-\pi (-Q_0^{-1} H Q_0^{-1} + 1/2 Q_0^{-1}HQ_0^{-1} H Q_0^{-1} + o(\lVert H\rVert))[m]t}= 1 + \pi Q_0^{-1} H Q_0^{-1} [m] + o(\lVert H\rVert)   ,
\end{equation}
we have the following contribution:
\begin{multline}
\sum_{m\in \mathbf{Z}^n\setminus \lbrace 0\rbrace} \int_1^\infty e^{-\pi Q_0^{-1}[m]t}  \, (1 + \pi Q_0^{-1} H Q_0^{-1}[m] t + o(\lVert H\rVert))\dfrac{dt}{t}\\
= \sum_{m\in \mathbf{Z}^n\setminus \lbrace 0\rbrace} \int_1^\infty e^{-\pi Q_0^{-1}[m] t } \, \dfrac{dt}{t} + \sum_{m\in \mathbf{Z}^n\setminus \lbrace 0\rbrace} \int_1^\infty e^{-\pi Q_0^{-1}[m]t} \pi (Q_0^{-1} H Q_0^{-1})[m] \, dt + o(\lVert H\rVert).
\end{multline}
Thus we get
\begin{multline}
F(e_{Q_0(H)}, 0)= F(Q_0,0) - \sum_{m\in \mathbf{Z}^n\setminus \lbrace 0\rbrace} H[m] \int_1^\infty \pi e^{-\pi Q_0[m]t} t^{n/2} \dfrac{dt}{t} 
\\+ \sum_{m\in \mathbf{Z}^n\setminus \lbrace 0\rbrace} (Q_0^{-1}HQ_0^{-1})[m]  \int_1^\infty e^{-\pi (Q_0^{-1}HQ_0^{-1})[m]t} + o(\lVert H\rVert)
\end{multline}
which is
\begin{equation}
F(Q_0, 0) + \sum_{m\in \mathbf{Z}^n\setminus \lbrace 0\rbrace} [- \psi_1(m) H[m] + \psi_2(m) (Q_0^{-1} HQ_0^{-1})[m]] + o(\lVert H\rVert)
\end{equation}
where
\begin{equation}
\psi_1(m)= \int_1^\infty \pi e^{-\pi Q_0 [m] t} t^{n/2}\, dt >0 \quad \text{and} \quad \psi_2(m)= \dfrac{e^{-\pi Q_0^{-1}[m]}}{Q_0^{-1}[m]} >0.
\end{equation}
We rewrite the first order term as a scalar product with $H$, so that the gradient is put in evidence:  
\begin{multline}\label{eq1}
\langle \Grad F(Q_0, 0), H\rangle = \langle H, \sum_{m\in\Z^n\setminus\lbrace 0\rbrace} \Big[-\psi_1(m) mm^t + \psi_2(m)(Q_0^{-1} m)(Q_0^{-1} m)^t \Big]\rangle .
\end{multline} 
We rearrange the terms of the sum in \eqref{eq1} according to the layers of $Q_0$ and $Q_0^{-1}$, and since $\psi_1(m)$ and $\psi_2(m)$ actually depend on $Q_0[m]$ and $Q_0^{-1}[m]$, we get
\begin{equation}\label{eq2}
\langle \Grad F(Q_0, 0), H\rangle =\langle H, -\sum_{k=1}^{\infty} \alpha_k \sum_{m\in M_k(Q_0)} m m^t +\sum_{j=1}^{\infty} \beta_j \sum_{m\in M_j(Q_0^{-1})}Q_0^{-1} m m^t Q_0^{-1}   \rangle
\end{equation} 
for positive $\alpha_k$ and $\beta_j$, namely 
\begin{equation}\label{eq4}
\alpha_k = \int_1^\infty \pi e^{-\pi m_k(Q_0) t} t^{n/2} \, dt\quad\text{and}\quad \beta_j=\dfrac{e^{- \pi m_j(Q_0^{-1})}}{m_j(Q_0^{-1})}. 
\end{equation}
We notice that since $H$ is orthogonal to $Q_0^{-1}$ by hypothesis, the scalar product in \eqref{eq2} is zero if and only if the second member is parallel to $Q_0^{-1}$. Thus we have proved the following
\begin{theor}
Let $\Lambda= A\Z^n$ be a lattice of covolume $1$ as above, and let $Q= A A^{t}$ be a Gram matrix for $\Lambda$. Then $\Lambda$ is a stationary point for the height function $h$ if and only if there exists $\lambda \in \R$ such that
\begin{equation}
-\sum_{k=1}^{\infty} \alpha_k \sum_{m\in M_k(Q_0)} m m^t +\sum_{j=1}^{\infty} \beta_j \sum_{m\in M_j(Q_0^{-1})}Q_0^{-1} m m^t Q_0^{-1} = \lambda Q_0^{-1}, 
\end{equation}
where the $\alpha_k$ and $\beta_j$ are positive constants given by \eqref{eq4}.
\end{theor}
To get Theorem \ref{th1}, suppose first that both $M_k(Q_0)$ and $M_j(Q_0^{-1})$ are spherical $2$-designs for every $k$ (resp. $j$), then by Corollary \ref{cor1} we have that for every $k$,
\begin{equation}
\sum_{m\in M_k(Q_0)} m m^t = c_k Q_0^{-1}
\end{equation}
and for every $j$  
\begin{equation}
\sum_{m\in M_j(Q_0^{-1})}Q_0^{-1} m m^t Q_0^{-1} = Q_0^{-1}\Big( \sum_{m\in M_j(Q_0^{-1})} m m^t \Big) Q_0^{-1} = Q_0^{-1} d_j Q_0 Q_0^{-1} = d_j Q_0^{-1}. 
\end{equation}
This implies that \eqref{eq2} becomes
\begin{equation}
\langle \Grad F(Q_0, 0), H\rangle =\langle H, - \sum_{k=1}^{\infty} \alpha_k c_k Q_0^{-1} + \sum_{j=1}^{\infty} \beta_j d_j Q_0^{-1}\rangle. 
\end{equation}
Since $H$ is orthogonal to $Q_0^{-1}$ and the second member is parallel to $Q_0^{-1}$, the product $\langle \Grad F(Q_0, 0), H\rangle$ is zero, i.e. $Q_0$ is a critical point. 
 
The conclusion is a consequence of the following lemma:  
\begin{lem}\label{lem3}
Let $\Lambda$ (resp. $Q$) be such that every layer $M_k (\Lambda)$ (resp. $M_k(Q)$) is a spherical $t$-design. Then also $\Lambda^*$ (resp. $Q^{-1}$) has a $t$-design on every layer.
\end{lem}
\begin{proof}
We prove the lemma in the language of lattices, and it is a consequence of the Poisson formula. Assume that $ \Lambda$ has a $t$-design on every layer: by Corollary \ref{cor2}, this is equivalent to
\begin{equation}
\theta_{\Lambda, P}(\tau) = \sum_{x\in \Lambda} P(x) e^{\pi i \tau (x,x)} \equiv 0 
\end{equation}  
for every homogeneous harmonic $P$ of degree $r= 1,\dots, t $. Then by the Poisson formula
\begin{equation}
\theta_{\Lambda^*, P}(\tau) = \theta_{\Lambda, P}\left(-\dfrac{1}{\tau}\right) \Bigg(\sqrt{\dfrac{i}{\tau}}\, \Bigg)^{-(n+2r)} i^{r} \vol(\R^{n}/\Lambda) =0
\end{equation} 
for $\tau\in \mathbf{H}$, and this implies that also $\Lambda^*$ has a $t$-design on every layer. 
\end{proof}
\section{Lattices which hold a $2$-design on every layer}\label{sec2}
As opposed to the case of $4$-designs, it is well known that lattices with $2$-designs on every layer do exist in every dimension $n$. Indeed, for instance all irreducible root lattices have this property, as was already noticed by Ry\v skov \cite{R73}, although in a different language. To see why this is true, we first define the automorphism group $\Aut(S)$ of a finite set $S$ of vectors of the same norm in $\R^n$ as
$$ \Aut(S)= \lbrace g \in \mathrm{O}(n) \ |\ g(S)\subseteq S \rbrace.$$

With this definition, the automorphism group $\Aut(\Lambda)$ of a lattice $\Lambda$ is obviously a subgroup of the automorphism group of each of its layers.

Then we have the following
\begin{theor}[cfr. \cite{Mar}, Theorem 3.6.6]\label{th3}
Let $S$ be a configuration of vectors of the same norm in $\R^n$, and $G$ a subgroup of $\Aut(S)$, with the property that $G$ acts irreducibly on $\R^n$. Then $S$ is a $2$-design.   
\end{theor}
Now, if $\Lambda$ is an irreducible root lattice in $\R^n$, then its Weyl group $W(\Lambda)$ (the subgroup of $\Aut(\Lambda)$ generated by the orthogonal reflections in the minimal vectors) acts irreducibly on $\R^n$ (see \cite{Eb}, Lemma 1.9), hence to conclude it suffices to apply Theorem \ref{th3} to each layer of $\Lambda$.
  As a consequence of this, and of Lemma \ref{lem3} we have this
\begin{cor}
The irreducible root lattices $A_n$, $D_n$, $E_6$, $E_7$ and $E_8$, as well as their dual lattices, are critical points for the height function. 
\end{cor}  
Since $E_6$, $E_7$ and $E_8$ actually have $4$-designs on every shell, it was already known, from \cite{Cou1}, that these lattices are local minima of the height function.   

Once established that such lattices, which we will call \emph{fully critical}, exist in every dimension, one would naturally like to ask for a complete classification in every given dimension, and for a test which allows to determine whether a given lattice is fully critical or not. The classification is intended up to similarity: we recall that a \emph{similarity of $\R^n$ of ratio $\lambda \neq 0$} is an endomorphism $\sigma$ of $\R^n$ which satisfies the identity
\begin{equation}
\sigma(x) \cdot \sigma(y) = \lambda^2 (x\cdot y) \quad\text{for all $x, y\in \R^n$}.
\end{equation}

For strongly eutactic lattices, the classification up to similarity is complete in dimension up to $6$, and it has been done successively by several mathematicians, namely \v{S}togrin, Berg\'e and Martinet for dimensions $2$ to $4$ (see \cite{Mar}, Sections 9.3 and 14.3), Batut for dimension $5$ (in \cite{Bt}), Elbaz-Vincent, Gangl and Soul\'e for dimension $6$. 

We recall two important facts:
\begin{prop}[\cite{Mar}, Theorem 9.4.3]
Up to similarity, there are only finitely many strongly eutactic lattices in a given dimension.
\end{prop}
\begin{prop}[\cite{MV}, Proposition 1.8]
A strongly eutactic lattice is proportional to an integral lattice.
\end{prop}
In the web page on lattices \cite{BM} created by Martinet and Batut, one can find a complete list of strongly eutactic lattices up to dimension $6$, together with many other known strongly eutactic lattices in higher dimension. Lattices there are given by an (integral) Gram matrix. Since fully critical lattices are in particular strongly eutactic, the representatives of the classes of equivalence of fully critical lattices up to similarity are to be found in this list. We are lead to distinguish the case of even dimension from that of odd dimension. 
\subsection{The case of even dimension}
We begin with even dimension $n$, and we assume $\Lambda$ to be (integral) and even; in terms of the matrix $Q$, this means that the diagonal of $Q$ is even.  If $\Lambda$ is not even, to our purposes it suffices to consider the matrix $2Q$: indeed, this is the Gram matrix of the rescaled lattice $\sqrt{2}\Lambda$, and the $2$-design property of the layers is invariant under homothety. Then Proposition \ref{prop5} applies, and for $P$ harmonic of degree $r$, $\theta_{\Lambda, P}$ is a modular form of weight $  n/2 + r$ with respect to $\Gamma_1(\ell)$, where $\ell$ is the level of $\Lambda$.

The problem of computing a basis of $q$-expansions for $\mathcal{M}_{k} (\Gamma_1(N))$, for any $k$ and $N$, can be solved algorithmically by a computer algebra system, for example Magma. Moreover, Magma gives this basis in upper triangular form with first Fourier coefficient $1$, up to a fixed precision. Therefore suppose that a basis for $\mathcal{M}_{k} (\Gamma_1(N))$ is given by
\begin{equation}
\begin{array}{cccccc}
q^{k_{1,1}} & + a_{1,2} q^{k_{1,2}} & +a_{1,3} q^{k_{1,3}} &+ \cdots & & \\
 &    & q^{k_{2,1}} &  +a_{2,2}q^{k_{2,2}} &+a_{2,3}q^{k_{2,3}} &+ \cdots\\
 & & \ddots & & &     \\
 & & & q^{k_{\dim M, 1}} & +a_{\dim M, 2}q^{k_{\dim M, 2}} &  + \cdots 
\end{array}
\end{equation}
Now $\theta_{\Lambda, P}$, for every given $P$ of degree $d$, is a linear combination of the elements of this basis, say
\begin{multline}
\theta_{\Lambda, P} = c_1(q^{k_{1,1}} + a_{1,2} q^{k_{1,2}} +a_{1,3} q^{k_{1,3}} + \cdots ) + c_2(q^{k_{2,1}}  +a_{2,2}q^{k_{2,2}} +a_{2,3}q^{k_{2,3}} + \cdots) \\+ \cdots + c_{\dim M} (q^{k_{\dim M, 1}}  +a_{\dim M, 2}q^{k_{\dim M, 2}} + \cdots).
\end{multline}
In particular, the coefficient of $q^{k_{1,1}}= e^{2 \pi i \tau k_{1,1} }$ is $c_1$. Suppose that the layer of $\Lambda$ given by $(x,x)= 2 k_{1,1}$ is a $t$-design, with $t$ even: then since
\begin{equation}
c_1 = \sum_{(x,x)=2 k_{1,1}} P(x)
\end{equation} 
and $P$ is harmonic, we have that $c_1=0$ by Proposition \ref{prop3}. Then we can drop it, and we get
\begin{multline}
\theta_{\Lambda, P} = c_2(q^{k_{2,1}}  +a_{2,2}q^{k_{2,2}} +a_{2,3}q^{k_{2,3}} + \cdots) + \cdots \\+ c_{\dim M} (q^{k_{\dim M, 1}}  +a_{\dim M, 2}q^{k_{\dim M, 2}} + \cdots).
\end{multline}
 Iterate the argument with the layer $(x,x)= 2k_{2,1}$ of $\Lambda$, to deduce that $c_2=0$, and then with the other coefficients. As a result, $\theta_{\Lambda, P}$ is the zero form, for every $P$ of even degree $t$. If we take $t=2$, this is equivalent to the fact that $\Lambda$ has a $2$-design on every layer, since the condition
 $\theta_{\Lambda, P}\equiv 0$ for $P$ of degree $1$ is always verified. In other words, we have proved the following
\begin{prop}\label{prop9}
Let $\Lambda$ be an even lattice in dimension $n$, with $n$ even, of level $\ell$, and let $\mathcal{B}$ be a basis of the space of modular forms $\mathcal{M}_{n/2+2} (\Gamma_1(\ell))$, in row echelon form with respect to the powers of the $q$-expansion. Let $N$ be the exponent of $q$ in the last row pivot. If the first $N$ layers of $\Lambda$ (i.e. the layers up to $(x,x)=2N$) are spherical $2$-designs, then all the layers of $\Lambda$ are.   
\end{prop}
Remark that the condition in Proposition \ref{prop9} is stronger than what is actually needed, because by the discussion above it is clear that it is sufficient to check the $\dim M$ layers of equation $(x,x)= 2k_{1,1}, \dots, 2k_{\dim M, 1}$.

Proposition \ref{prop9} suggests the following algorithm to check whether a given lattice $\Lambda$ has $2$-designs on every layer:

\begin{algorithm}
\caption{Given $\Lambda$, it tells whether $\Lambda$ is fully critical or not}
\label{algo_fully}
\begin{algorithmic}[1]
\REQUIRE an even lattice $\Lambda$ of even dimension $n$
\ENSURE ``$\Lambda$ has $2$-designs on every layer'' OR ``FAILURE at the layer $k$''
\STATE compute the level $\ell$ of $\Lambda$;
\STATE compute the Gram matrix $Q$;
\STATE compute a basis of $\mathcal{M}_{n/2+2}(\Gamma_1(\ell))$;
\STATE put it in row echelon form;
\STATE find the exponent $N$ of $q$ in the last pivot;
\STATE find the layers of $Q$ up to $(x\cdot x)= 2N$ (for example, via LLL);
\STATE for $k=1$ to $2N$, test the $2$-design property on the layer $(x\cdot x)=k$ via Proposition \ref{prop3};
\STATE If they are all $2$-designs, conclude that $\Lambda$ has $2$-designs on every layer, otherwise FAILURE.
\end{algorithmic}
\end{algorithm}
 
Our computations, which have been performed with Magma and Pari/GP, are given in Appendix \ref{ApB}. Here we summarise the results in the following
\begin{prop}\label{prop7}
Up to similarity, there are:
\begin{enumerate}[(i)]
 \item $2$ strongly eutactic lattices, both fully critical, in dimension $n=2$;
\item  $6$ strongly eutactic lattices, all fully critical, in dimension $4$;
\item  $17$ fully critical lattices, out of $21$ strongly eutactic, in dimension $6$.
\end{enumerate}
\end{prop}
\subsection{The case of odd dimension}
If the dimension $n$ of $\Lambda$ is odd, we cannot apply Proposition \ref{prop5} directly, so the strategy consists in considering an auxiliary lattice $\Lambda_2$
of dimension $n+1$. We have the following Lemma, which holds in a general situation; probably it is already known, but we have not found it in the literature.
\begin{lem}\label{lem2}
Let $\Lambda_1$ and $\Lambda_2$ be two lattices of dimension $n_1$ and $n_2$ respectively. Let $P_1=P_1(x_1, \dots, x_{n_1})\in\R[X_1, \dots, X_{n_1}]$ be a harmonic polynomial for the $n_1$-dimensional Laplacian $\Delta_{n_1} = \pd^2/\pd x_1^2 + \cdots +\pd ^2 / \pd x_{n_1}^2 $ and $P_2=P_2(y_1, \dots, y_{n_2})$ be a harmonic polynomial for the $n_2$-dimensional Laplacian $\Delta_{n_2} = \pd^2/\pd y_1^2 + \cdots +\pd ^2 / \pd y_{n_2}^2 $. Then the polynomial $$P= P(x_1, \dots, x_{n_1}, y_1, \dots, y_{n_2})=P_1(x_1, \dots, x_{n_1})P_2(y_1, \dots, y_{n_2}) $$ is harmonic for the $(n_1 + n_2)$-dimensional Laplacian $\Delta_{n_1+ n_2} = \pd^2/\pd x_1^2 + \cdots +\pd ^2 / \pd x_{n_1}^2 +  \pd^2/\pd y_1^2 + \cdots +\pd ^2 / \pd y_{n_2}^2 $. Moreover, if the lattice $\Lambda := \Lambda_1 \perp \Lambda_2$ is the $(n_1 + n_2)$-dimensional orthogonal direct sum of $\Lambda_1$ and $\Lambda_2$, then
\begin{equation}
\theta_{\Lambda, P}(\tau)= \theta_{\Lambda_1, P_1}(\tau) \theta_{\Lambda_2, P_2}(\tau)
\end{equation}
\end{lem} 
\begin{proof}
The fact that $P$ is harmonic for $\Delta_{n_1 + n_2} $ comes directly from the identity 
\begin{equation}
\Delta_{n_1+n_2} P= (\Delta_{n_1 }P_1) P_2 + P_1 (\Delta_{n_22} P_2) 
\end{equation} 
and the multiplicativity of the theta function is due to the orthogonal direct sum, as every $X\in \Lambda$ can be written uniquely as $x + y$, with $x\in \Lambda_1$ and $y\in \Lambda_2$, $(x\cdot y)=0$:
\begin{equation}\begin{split}
\theta_{\Lambda, P}(\tau) &= \sum_{X \in \Lambda} P(X) e^{ \pi i \tau (X \cdot X)}\\
&= \sum_{x\in \Lambda_1, y\in \Lambda_2} P_1(x)P_2(y) e^{\pi i \tau(x\cdot x) }e^{\pi i \tau (y\cdot y)} =\theta_{\Lambda_1, P_1}(\tau) \theta_{\Lambda_2, P_2}(\tau),
\end{split}\end{equation}
\end{proof}

 We apply Lemma \ref{lem2} to the lattice 
\begin{equation}
\Lambda' := \Lambda \perp A_1
\end{equation} 
(or, as we did above, $\sqrt{2}\Lambda \perp A_1$ if $\Lambda$ is not even) and to the harmonic polynomials $P_1(x_1, \dots, x_n)$ of degree $2$ for $\Lambda$, arbitrarily chosen, and $P_2 =1$ for $A_1$. Before going further, notice that in terms of quadratic forms, it amounts to consider the Gram matrix
\begin{equation}
\left(
\begin{array}{c|c}
   \raisebox{-15pt}{{\huge\mbox{{$Q$}}}} & 0 \\[-4ex]
   & \vdots \\[-0.5ex]
   & 0 \\
  \hline
  0\cdots 0 & 2  
\end{array}
\right)
\end{equation}
(put $2Q$ if the diagonal of $Q$ is not even; from now on we assume that $\Lambda$ is even).

By the choice of the constant polynomial for $P_2$, the product $P=P_1P_2$ is still of degree $2$, and Lemma \ref{lem2} yields
\begin{equation}
\theta_{\Lambda', P}(\tau)=\theta_{\Lambda, P_1}(\tau) \theta_{A_1}(\tau), 
\end{equation}  
where it is important to remark that
\begin{equation}
\theta_{A_1}(\tau)= \sum_{m\in \Z}e^{\pi i \tau \cdot 2 m^2 } = 1 + 2\sum_{m=1}^\infty q^{m^2} = 1 + 2q + 2q^4 + 2q^9 + \cdots 
\end{equation}
is not identically zero. Now we can apply Proposition \ref{prop5} to $\theta_{\Lambda', P}$, of even dimension $n+1$, and we have that 
\begin{equation}
\theta_{\Lambda', P}(\tau) \in \mathcal{M}_{\frac{n+1}{2} + 2}(\Gamma_1(\ell'))
\end{equation}
where $\ell'$ is the level of $\Lambda'$. We argue as in the case of even $n$: suppose that a row echelon basis for $\mathcal{M}_{\frac{n+1}{2} + 2}(\Gamma_1(\ell'))$ is given by
\begin{equation}
\begin{array}{cccccc}
q^{k_{1,1}} & + a_{1,2} q^{k_{1,2}} & +a_{1,3} q^{k_{1,3}} &+ \cdots & & \\
 &    & q^{k_{2,1}} &  +a_{2,2}q^{k_{2,2}} &+a_{2,3}q^{k_{2,3}} &+ \cdots\\
 & & \ddots & & &     \\
 & & & q^{k_{\dim M, 1}} & +a_{\dim M, 2}q^{k_{\dim M, 2}} &  + \cdots 
\end{array}
\end{equation}
and call $N= k_{\dim M, 1}$. If the layers of $\Lambda$ up to $(x\cdot x) = 2N$ are $2$-designs, then by Proposition \ref{prop6} the first $N$ coefficients of $\theta_{\Lambda, P_1}$ are zero, that is $\theta_{\Lambda, P_1} = a_{N+1} q^{N+1} + a_{N+2}q^{N+2} + \cdots$. Therefore 
\begin{equation}\begin{split}
\theta_{\Lambda', P} &= (a_{N+1} q^{N+1} + a_{N+2}q^{N+2} + \cdots)(1 + 2q + 2q^4 + \cdots)\\
&= a_{N+1}q^{N+1} + (2a_{N+1} + a_{N+2})q^{N+2} + \cdots
\end{split}\end{equation}
and by the same argument as for even $n$, we deduce that the only possibility for $\theta_{\Lambda', P}$ to be a combination of the elements of the basis is being zero.
As $\theta_{\Lambda', P} = \theta_{\Lambda, P_1} \theta_{A_1}$, and $\theta_{A_1}$ is not identically zero, this implies that $\theta_{\Lambda, P_1}$ is identically zero. Since $P_1$ was arbitrary, by  Corollary \ref{cor2} we conclude that every other layer of $\Lambda$ is a $2$-design.

We summarise all of this in the following 
\begin{prop}
Let $\Lambda$ be an even lattice in dimension $n$, with $n$ odd, and consider the lattice $\Lambda' = \Lambda \perp A_1$. Let $\ell'$ be the level of $\Lambda'$, and $\mathcal{B}'$ a basis of the space of modular forms $\mathcal{M}_{(n+1)/2 +2} (\Gamma_1(\ell') )$, in row echelon form with respect to the powers of the $q$-expansion. Let $N$ be the exponent of $q$ in the last row pivot. If the layers of $\Lambda$ up to $(x\cdot x)= 2N$ are spherical $2$-designs, then all the layers of $\Lambda$ are.
\end{prop}  

As for even $n$, we refer to Appendix \ref{ApB} for the complete list of fully critical lattices in dimension $3$ and $5$, and we summarise the results in the following
\begin{prop}\label{prop8}
Up to similarity, there are
\begin{enumerate}[(i)]
\item $3$ strongly eutactic lattices in dimension $n=3$, all fully critical.
\item $7$ fully critical lattices out of $9$ strongly eutactic, in dimension $5$.
\end{enumerate}
\end{prop}
Although not complete, we have added also the list of fully critical lattices in dimension $7$, taken from Martinet's list of strongly eutactic lattices. According to our computations, there are $7$ fully eutactic lattices out of the $17$ listed there.  
\section{Comparison with known results and final remarks}\label{sec4}
We conclude with some remarks and open questions.
\subsection{}
Theorem \ref{th1} only gives a sufficient condition for a lattice $\Lambda$ to be a stationary point for the height. We do not know yet whether having a $2$-design on every layer is a necessary condition for a lattice to be critical. Moreover, as our first aim is to study the minima of the height function, if the $2$-design condition is not necessary, we would like to know whether the minimum of $h$ is to be found among the fully critical lattices. This is true for the hexagonal lattice $A_2$ in dimension $2$, and for the face-centred cubic lattice $A_3^*$ in dimension $3$; this is also true for $D_4$, which achieves a strict local minimum of $h$ in dimension $4$.  
\subsection{}
Moreover, Theorem \ref{th1} does not tell whether a fully critical lattice is a local maximum or minimum, or a saddle point for $h$.
\subsection{}  
We noticed a quite intriguing property of strongly eutactic lattices: up to dimension $6$, if a strong eutactic lattice $\Lambda$ is not fully critical, then the sequence of $2$-designs already stops at the second layer, i.e. $M_2(\Lambda)$ is not a $2$ design. As we found no counterexample among the strongly eutactic lattices at our disposal in dimension $7$ and $8$, we are lead to the following
\begin{conj}
If the first two layers $M_1(\Lambda)$ and $M_2(\Lambda)$ of a lattice $\Lambda$ are spherical $2$-designs, then all the other layers are $2$-designs, in particular $\Lambda$ is fully critical for $h$. 
\end{conj}     
The conjecture is perhaps hazardous in high dimension, but thanks to the complete classification of strongly eutactic lattices in small dimension, we can state it as a theorem in dimension up to $6$:
\begin{theor}
Let $\Lambda$ be a lattice of dimension $2\leq n \leq 6$. If its first two layers $M_1(\Lambda)$ and $M_2(\Lambda)$ are spherical $2$-designs, then all the other layers are $2$-designs, in particular $\Lambda$ is fully critical for $h$.  
\end{theor}

\section*{Acknowledgements}
Experiments presented in this paper were carried out using the PLAFRIM experimental testbed, being developed under the Inria PlaFRIM development action with support from LABRI and IMB and other entities: Conseil R\'egional d'Aquitaine, FeDER, Universit\'e de Bordeaux and CNRS (see https://plafrim.bordeaux.inria.fr/).
\appendix
\section{An explicit example}
To clarify the computations described in Section \ref{sec2}, we put here a concrete example of such computations. We choose the $6$-dimensional lattice $ \Lambda =  \mathrm{Ext}^2(A_4) = ste10a $ of Martinet's list. Its Gram matrix is
\begin{equation}
Q=\begin{pmatrix} 3 & 1 & 1 & 1 & 1 & 0 \\ 1 & 3 & -1 & 1 & 0 & 1 \\ 1 & -1 & 3 & 0 & 1 & -1\\ 1 & 1 & 0 & 3 & -1 & -1 \\ 1 & 0 & 1 & -1 & 3 & 1 \\ 0 & 1 & -1 & -1 & 1 & 3 \end{pmatrix}
\end{equation}
which is odd. We take then the matrix
\begin{equation}
2Q = \begin{pmatrix} 6 & 2 & 2 & 2 & 2 & 0 \\ 2 & 6 & -2 & 2 & 0 & 2 \\ 2 & -2 & 6 & 0 & 2 & -2\\ 2 & 2 & 0 & 6 & -2 & -2 \\ 2 & 0 & 2 & -2 & 6 & -2 \\ 0 & 2 & -2 & -2 & 2 & 6 \end{pmatrix}
\end{equation}
and we work with the even lattice $ \sqrt{2}\Lambda $ which has $ 2Q $ as Gram matrix. We ask Magma to compute the level of $ \sqrt{2}\Lambda $, which is $ 20 $, and a row echelon basis of   $\mathcal{M}_{6/2 +2} (\Gamma_1(20) )$. The answer is: 
\begin{verbatim}
Space of modular forms on Gamma_1(20) of weight 5 and dimension 58 over Integer 
Ring.
[
    1 + 31159607208500*q^57 - 232196944448000*q^58 + 507270570810000*q^59 + 
    1167922181810300*q^61 + O(q^62),
    q + 382543449223294*q^57 - 3514573787260680*q^58 + 11337240787683570*q^59 - 
    134706130818436118*q^61 + O(q^62),
    q^2 + 548054839647181*q^57 - 5065849016146885*q^58 + 16478465006713624*q^59 
    - 199877223135259773*q^61 + O(q^62),
    ...
    q^55 - 75*q^57 + 435*q^58 - 870*q^59 + 341*q^61 + O(q^62),
    q^56 - 11*q^57 + 46*q^58 - 72*q^59 - 409*q^61 + O(q^62),
    q^60 - 12*q^61 + O(q^62)
]

\end{verbatim}
The exponent $ N $ we are looking for is $ 60 $, hence we have to do the $ 2 $-design test on the layers of $ \sqrt{2}\Lambda $ up to $ (x\cdot x) = 120 $; equivalently, we can test the layers of $ \Lambda $ up to $ (x\cdot x)=60 $. Pari gives the following answer:
\begin{verbatim}
? \r layers.gp;
  ***   Warning: new stack size = 1600000000 (1525.879 Mbytes).
[3, 1, 1, 1, 1, 0; 1, 3, -1, 1, 0, 1; 1, -1, 3, 0, 1, -1; 
1, 1, 0, 3, -1, -1; 1, 0, 1, -1, 3, 1;
0, 1, -1, -1, 1, 3] 62
[0, 0, 10, 15, 12, 30, 30, 30, 60, 72, 120, 100, 60, 150, 190,
 195, 120, 200, 360, 267, 300, 240, 390, 510, 324, 600, 540, 420, 
 480, 630, 960, 510, 480, 960, 1050, 1125, 540, 720, 1680, 1122, 1020, 
 1160, 1170, 1560, 1152, 1590, 1470, 1420, 1500, 1704, 2880, 1440, 1140, 2460, 
 2520, 2550, 1440, 1920, 3480, 2620, 2220, 1920]
the layer (x,x)=1 is empty
the layer (x,x)=2 is empty
150 = 150, 2-DESIGN on the layer (x,x)=3
600 = 600, 2-DESIGN on the layer (x,x)=4
600 = 600, 2-DESIGN on the layer (x,x)=5
5400 = 5400, 2-DESIGN on the layer (x,x)=6
7350 = 7350, 2-DESIGN on the layer (x,x)=7
9600 = 9600, 2-DESIGN on the layer (x,x)=8
48600 = 48600, 2-DESIGN on the layer (x,x)=9
86400 = 86400, 2-DESIGN on the layer (x,x)=10
...
1122854400 = 1122854400, 2-DESIGN on the layer (x,x)=57
2066841600 = 2066841600, 2-DESIGN on the layer (x,x)=58
7026050400 = 7026050400, 2-DESIGN on the layer (x,x)=59
4118640000 = 4118640000, 2-DESIGN on the layer (x,x)=60
\end{verbatim}
Therefore we can conclude that every layer of $ \Lambda $ (and of $\sqrt{2}\Lambda$) is a $ 2 $-design.

\section{Tables}\label{ApB}
\subsection{Even dimension $n=2, 4,6$.}
The first entry is the name of the lattice $\Lambda$ according to the classification of Martinet and Batut; the second entry is a Gram matrix $Q$ for $\Lambda$; if $\Lambda$ is even, we give the dimension $\dim M$ of the space $\mathcal{M}_{n/2+2} (\Gamma_1(\ell)) $ and the exponent $N$ in the pivot $q^N$. Finally we give the traditional name of $\Lambda$, if one exists. If $\Lambda$ is odd, the numbers $\dim M$ and $N$ refer to the lattice associated to $2Q$, that is $\sqrt{2}\Lambda$, which has the same design structure of $\Lambda$.
\begin{center}
\begin{tabular}{l | c | c | c |  l}
$\Lambda$ & $Q$ &  $\dim M $ & $N$  & Remarks\\
\hline
sta2 & $\begin{pmatrix}1 & 0 \\ 0 & 1 \end{pmatrix}$ & 2 & 1 & $\Z^2$\\
sta3 & $\begin{pmatrix}2 & 1 \\ 1 & 2 \end{pmatrix}$ & 2 & 1 & $A_2$\\
\end{tabular}
\end{center}
\begin{center}
\begin{tabular}{l | c | c | c |  l }
$\Lambda$ & $Q$ &  $\dim M $ & $N$ & Remarks\\
\hline
stc4 & $\begin{pmatrix} 1 & 0 & 0 & 0\\ 0 & 1 & 0 & 0\\ 0 & 0 & 1 & 0\\ 0 & 0 & 0 & 1\end{pmatrix}$ & 3 & 2 & $\Z^4$\\
stc5 & $\begin{pmatrix} 4 & -1 & -1 & -1\\ -1 & 4 & -1 & -1\\ -1 & -1 & 4 & -1\\ -1 & -1 & -1 & 4\end{pmatrix}$ & 5 & 4  & $A_2^*$\\
stc6 & $\begin{pmatrix} 2 & 1 & 0 & 0\\ 1 & 2 & 0 & 0\\ 0 & 0 & 2 & 1\\ 0 & 0 & 1 & 2\end{pmatrix}$ & 2 & 1 & $A_2\perp A_2 $\\
stc9 & $\begin{pmatrix} 4 & -2 & -2 & 1\\ -2 & 4 & 1 & -2\\ -2 & 1 & 4 & -2\\ 1 & -2 & -2 & 4\end{pmatrix}$ & 13 & 12 & $A_2\otimes A_2 $\\
stc10 & $\begin{pmatrix} 2 & 1 & 1 & 1\\ 1 & 2 & 1 & 1\\ 1 & 1 & 2 & 1\\ 1 & 1 & 1 & 2\end{pmatrix}$ & 5 & 4 & $A_4 = P_4^2 $\\
stc12 & $\begin{pmatrix} 2 & 0 & 1 & 1\\ 0 & 2 & 1 & 1\\ 1 & 1 & 2 & 1\\ 1 & 1 & 1 & 2\end{pmatrix}$ & 2& 1 & $D_4 = P_4^1 $\\
\end{tabular}
\end{center}

\begin{center}
\begin{tabular}{l | c | c | c |  l }
$\Lambda$ & $Q$ &  $\dim M $ & $N$ & Remarks\\
\hline
ste6a & $\begin{pmatrix} 1 & 0 & 0 & 0 & 0 & 0\\ 0 & 1 & 0 & 0 & 0 & 0\\ 0 & 0 & 1 & 0 & 0 & 0\\ 0 & 0 & 0 & 1 & 0 & 0 \\ 0 & 0 & 0 & 0 & 1 & 0 \\ 0 & 0 & 0 & 0 & 0 & 1\end{pmatrix}$ & 3 & 2 & $\Z^6$\\
ste6c & $\begin{pmatrix} 3 & 1 & 1 & 1 & 1 & 1\\ 1 & 2 & 0 & 0 & 0 & 0\\ 1 & 0 & 2 & 0 & 0 & 0\\ 1 & 0 & 0 & 2 & 0 & 0 \\ 1 & 0 & 0 & 0 & 2 & 0 \\ 1 & 0 & 0 & 0 & 0 & 2\end{pmatrix}$ & 3 & 2 & $D_6^*$\\
ste7 & $\begin{pmatrix} 6 & -1 & -1 & -1 & -1 & -1\\ -1 & 6 & -1 & -1 & -1 & -1 \\ -1 & -1 & 6 & -1 & -1 & -1 \\ -1 & -1 & -1 & 6 & -1 & -1 \\ -1 & -1 & -1 & -1 & 6 & -1 \\ -1 & -1 & -1 & -1 & -1 & 6 \end{pmatrix}$ & 11 & 10 & $A_6^*$\\
ste8 & $\begin{pmatrix} 3 & -1 & -1 & 0 & 0 & 0 \\ -1 & 3 & -1 & 0 & 0 & 0 \\ -1 & -1 & 3 & 0 & 0 & 0\\ 0 & 0 & 0 & 3 & -1 & -1 \\ 0 & 0 & 0 & -1 & 3 & -1 \\ 0 & 0 & 0 & -1 & -1 & 3 \end{pmatrix}$ & 11 & 10 & $A_3^* \perp A_3^*$\\
ste9 & $\begin{pmatrix}2 & 1 & 0 & 0 & 0 & 0 \\ 1 & 2 & 0 & 0 & 0 & 0\\ 0 & 0 & 2 & 1 & 0 & 0 \\ 0 & 0 & 1 & 2 & 0 & 0 \\ 0 & 0 & 0 & 0 & 2 & 1 \\ 0 & 0 & 0 & 0 & 1 & 2 \\ \end{pmatrix}$ & 2 & 1 & $A_2 \perp A_2 \perp A_2$ \\
ste10a & $\begin{pmatrix} 3 & 1 & 1 & 1 & 1 & 0 \\ 1 & 3 & -1 & 1 & 0 & 1 \\ 1 & -1 & 3 & 0 & 1 & -1 \\ 1 & 1 & 0 & 3 & -1 & -1 \\ 1 & 0 & 1 & -1 & 3 & 1 \\ 0 & 1 & -1 & -1 & 1 & 3 \end{pmatrix}$ & 58 & 60 & $\mathrm{Ext}^2(A_4)$\\
ste12a & $\begin{pmatrix} 2 & -1 & -1 & 0 & 0 & 0 \\ -1 & 2 & 1 & 0 & 0 & 0 \\ -1 & 1 & 2 & 0 & 0 & 0 \\ 0 & 0 & 0 & 2 & -1 & -1 \\ 0 & 0 & 0 & -1 & 2 & 1 \\ 0 & 0 & 0 & -1 & 1 & 2 \end{pmatrix}$ & 11 & 10 & $A_3 \perp A_3$\\
ste12b & $\begin{pmatrix} 5 & -2 & 2 & 2 & -2 & 1\\ -2 & 5 & 1 & -1 & -1 & 0 \\ 2 & 1 & 5 & -1 & -1 & 2\\ 2 & -1 & -1 & 5 & -1 & -2 \\ -2 & -1 & -1 & -1 & 5 & -2 \\ 1 & 0 & 2 & -2 & -2 & 5 \end{pmatrix}$ & 58 & 60 & $\mathrm{Ext}^2(A_4)_{\mathrm{even}}^*$\\
ste12c & $\begin{pmatrix} 6 & 3 & -2 & -1 & -2 & -1 \\ 3 & 6 & -1 & -2 & -1 & -2 \\ -2 & -1 & 6 & 3 & -2 & -1 \\ -1 & -2 & 3 & 6 & -1 & -2 \\ -2 & -1 & -2 & -1 & 6 & 3 \\ -1 & -2 & -1 & -1 & 3 & 6 \end{pmatrix} $ & 21  & 20 & $A_2 \otimes A_3^*$\\
\end{tabular}
\end{center}

\begin{center}
\begin{tabular}{l | c | c | c |  l }
$\Lambda$ & $Q$ &  $\dim M $ & $N$ & Remarks\\
\hline
ste15a & $\begin{pmatrix} 4 & -2 & -1 & 0 & 1 & -2 \\ -2 & 4 & -1 & -1 & -2 & 1 \\ -1 & -1 & 4 & -1 & 0 & 1 \\ 0 & -1 & -1 & 4 & 2 & 1 \\ 1 & -2 & 0 & 2 & 4 & -1 \\ -2 & 1 & 1 & 1 & -1 & 4 \end{pmatrix} $ & 58 & 60 & $\mathrm{Ext}^2(A_4)_{\mathrm{even}}$\\
ste16 & $\begin{pmatrix} 3 & 1 & 1 & -1 & 1 & 1 \\ 1 & 3 & -1 & -1 & -1 & 1 \\ 1 & -1 & 3 & -1 & 1 & -1 \\ -1 & -1 & -1 & 3 & -1 & -1 \\ 1 & -1 & 1 & -1 & 3 & 1 \\ 1 & 1 & -1 & -1 & 1 & 3 \end{pmatrix} $ & 39 & 40 & $D_6^+$\\
ste18a & $\begin{pmatrix} 4 & 2 & 2 & 1 & 2 & 1 \\ 2 & 4 & 1 & 2 & 1 & 2 \\ 2 & 1 & 4 & 2 & 2 & 1 \\ 1 & 2 & 2 & 4 & 1 & 2 \\ 2 & 1 & 2 & 1 & 4 & 2 \\ 1 & 2 & 1 & 2 & 2 & 4 \end{pmatrix} $ & 21  & 20 & $A_2 \otimes A_3$\\
ste21a & $\begin{pmatrix} 2 & 1 & 1 & 1 & 1 & 1 \\ 1 & 2 & 1 & 1 & 1 & 1 \\ 1 & 1 & 2 & 1 & 1 & 1 \\ 1 & 1 & 1 & 2 & 1 & 1 \\ 1 & 1 & 1 & 1 & 2 & 1 \\ 1 & 1 & 1 & 1 & 1 & 2 \end{pmatrix} $ & 11 & 10 & $A_6$\\
ste21b & $\begin{pmatrix} 4 & -2 & -2 & -1 & -2 & -2 \\ -2 & 4 & 2 & -1 & 1 & 2 \\ -2 & 2 & 4 & 1 & 2 & 1 \\ -1 & -1 & 1 & 4 & 0 & -1 \\ -2 & 1 & 2 & 0 & 4 & 0 \\ -2 & 2 & 1 & -1 & 0 & 4 \end{pmatrix} $ & 11  & 10 & $A_6^{(2)} = P_6^5$\\
ste27 & $\begin{pmatrix} 4 & -2 & -1 & 1 & 1 & 1 \\ -2 & 4 & -1 & -2 & -2 & -2 \\ -1 & -1 & 4 & -1 & -1 & -1 \\ 1 & -2 & -1 & 4 & 1 & 1 \\ 1 & -2 & -1 & 1 & 4 & 1 \\ 1 & -2 & -1 & 1 & 1 & 4 \end{pmatrix} $ & 2 & 1 & $E_6^2 = P_6^2$\\
ste30 & $\begin{pmatrix} 2 & 0 & 1 & 1 & 1 & 1 \\ 0 & 2 & 1 & 1 & 1 & 1 \\ 1 & 1 & 2 & 1 & 1 & 1 \\ 1 & 1 & 1 & 2 & 1 & 1 \\ 1 & 1 & 1 & 1 & 2 & 1 \\ 1 & 1 & 1 & 1 & 1 & 2 \end{pmatrix} $ & 3  & 2 & $D_6 = P_6^3$\\
ste36 & $\begin{pmatrix} 2 & 0 & 0 & 1 & 1 & 1 \\ 0 & 2 & 1 & 1 & 1 & 1 \\ 0 & 1 & 2 & 1 & 1 & 1 \\ 1 & 1 & 1 & 2 & 1 & 1 \\ 1 & 1 & 1 & 1 & 2 & 1 \\ 1 & 1 & 1 & 1 & 1 & 2 \end{pmatrix} $ & 3 & 2 & $E_6 = P_6^1$\\
\end{tabular}
\end{center}
\subsection{Odd dimension $n=3,5, 7$ (the last is incomplete).}
Here, if $\Lambda$ is even, the entry $\dim M$ is the dimension of $\mathcal{M}_{(n+1)/2 +2} (\Gamma_1(\ell') )$, where $\ell'$ is the level of $\Lambda'=\Lambda \perp A_1$, and $N$ is the exponent of the last pivot $q^N$. If $\Lambda$ is odd, $\dim M$ and $N$ refer to $\sqrt{2}\Lambda \perp A_1$. 
\begin{center}
\begin{tabular}{l | c | c | c |  l }
$\Lambda$ & $Q$ &  $\dim M $ & $N$ & Remarks\\
\hline
stb3 & $\begin{pmatrix} 1 & 0 & 0 \\ 0 & 1 & 0 \\ 0 & 0 & 1 \end{pmatrix} $ & 3 & 2 & $\mathbf{Z}^3$\\
stb4 & $\begin{pmatrix} 3 & -1 & -1 \\ -1 & 3 & -1 \\ -1 & -1 & 3 \end{pmatrix} $ & 9 & 8 & $ A_3^*$\\
stb6 & $\begin{pmatrix} 2 & 1 & 1 \\ 1 & 2 & 1 \\ 1 & 1 & 2 \end{pmatrix} $ & 9 & 8 & $A_3$\\
\end{tabular}
\end{center}
\begin{center}
\begin{tabular}{l | c | c | c |  l }
$\Lambda$ & $Q$ &  $\dim M $ & $N$ & Remarks\\
\hline
std5a & $\begin{pmatrix} 1 & 0 & 0 & 0 & 0 \\ 0 & 1 & 0 & 0 & 0 \\ 0 & 0 & 1 & 0 & 0 \\ 0 & 0 & 0 & 1 & 0 \\ 0 & 0 & 0 & 0 & 1 \end{pmatrix}$ & 3 & 2 & $ \mathbf{Z}^5 $\\
std5b & $\begin{pmatrix} 5 & 2 & 2 & 2 & 2 \\ 2 & 4 & 0 & 0 & 0 \\ 2 & 0 & 4 & 0 & 0 \\ 2 & 0 & 0 & 4 & 0 \\ 2 & 0 & 0 & 0 & 4\end{pmatrix}$ & 11 & 10 & $ D_5^* $\\
std6 & $\begin{pmatrix} 5 & -1 & -1 & -1 & -1 \\ -1 & 5 & -1 & -1 & -1 \\ -1 & -1 & 5 & -1 & -1 \\ -1 & -1 & -1 & 5 & -1 \\ -1 & -1 & -1 & -1 & 5\end{pmatrix}$ & 21 & 20 & $ A_5^* $\\
std10 & $\begin{pmatrix} 3 & -1 & -1 & -1 & 1 \\ -1 & 3 & -1 & -1 & -1 \\ -1 & -1 & 3 & 1 & -1 \\ -1 & -1 & 1 & 3 & -1 \\ 1 & -1 & -1 & -1 & 3\end{pmatrix}$ & 21 & 20 & $ A_5^2 $\\
std15a & $\begin{pmatrix} 4 & -2 & -1 & -2 & 1 \\ -2 & 4 & -1 & 1 & -2 \\ -1 & -1 & 4 & -1 & -1 \\ -2 & 1 & -1 & 4 & 1 \\ 1 & -2 & -1 & 1 & 4\end{pmatrix}$ & 21 & 20 & $ A_5^2 $\\
std15b & $\begin{pmatrix} 2 & 1 & 1 & 1 & 1 \\ 1 & 2 & 1 & 1 & 1 \\ 1 & 1 & 2 & 1 & 1 \\ 1 & 1 & 1 & 2 & 1 \\ 1 & 1 & 1 & 1 & 2\end{pmatrix}$ & 21 & 20 & $ A_5^2 $\\
std20 & $\begin{pmatrix} 2 & 0 & 1 & 1 & 1 \\ 0 & 2 & 1 & 1 & 1 \\ 1 & 1 & 2 & 1 & 1 \\ 1 & 1 & 1 & 2 & 1 \\ 1 & 1 & 1 & 1 & 2\end{pmatrix}$ & 11 & 10 & $ A_5^2 $\\
\end{tabular}
\end{center}

\begin{center}
\begin{tabular}{l | c | c | c |  l }
$\Lambda$ & $Q$ &  $\dim M $ & $N$ & Remarks\\
\hline
stf7a & $\begin{pmatrix} 1 & 0 & 0 & 0 & 0 & 0 & 0 \\ 0 & 1 & 0 & 0 & 0 & 0 & 0 \\ 0 & 0 & 1 & 0 & 0 & 0 & 0 \\ 0 & 0 & 0 & 1 & 0 & 0 & 0 \\ 0 & 0 & 0 & 0 & 1 & 0 & 0\\ 0 & 0 & 0 & 0 & 0 & 1 & 0 \\ 0 & 0 & 0 & 0 & 0 & 0 & 1  \end{pmatrix}$ & 4 & 3 & $ \mathbf{Z}^7 $\\
stf7d & $\begin{pmatrix} 7 & 2 & 2 & 2 & 2 & 2 & 2 \\ 2 & 4 & 0 & 0 & 0 & 0 & 0 \\ 2 & 0 & 4 & 0 & 0 & 0 & 0 \\ 2 & 0 & 0 & 4 & 0 & 0 & 0 \\ 2 & 0 & 0 & 0 & 4 & 0 & 0\\ 2 & 0 & 0 & 0 & 0 & 4 & 0 \\ 2 & 0 & 0 & 0 & 0 & 0 & 4  \end{pmatrix}$ & 13 & 12 & $ D_7^* $\\
stf8 & $\begin{pmatrix} 7 & -1 & -1 & -1 & -1 & -1 & -1 \\ -1 & 7 & -1 & -1 & -1 & -1 & -1 \\ -1 & -1 & 7 & -1 & -1 & -1 & -1 \\ -1 & -1 & -1 & -7 & -1 & -1 & -1 \\ -1 & -1 & -1 & -1 & 7 & -1 & -1\\ -1 & -1 & -1 & -1 & -1 & 7 & -1 \\ -1 & -1 & -1 & -1 & -1 & -1 & 7  \end{pmatrix}$ & 47 & 46 & $ A_7^* $\\
stf28a & $\begin{pmatrix} 2 & 1 & 1 & 1 & 1 & 1 & 1 \\ 1 & 2 & 1 & 1 & 1 & 1 & 1 \\ 1 & 1 & 2 & 1 & 1 & 1 & 1 \\ 1 & 1 & 1 & 2 & 1 & 1 & 1 \\ 1 & 1 & 1 & 1 & 2 & 1 & 1\\ 1 & 1 & 1 & 1 & 1 & 2 & 1 \\ 1 & 1 & 1 & 1 & 1 & 1 & 2  \end{pmatrix}$ & 47 & 46 & $ A_7 $\\
stf28b & $\begin{pmatrix} 3 & 1 & 1 & -1 & 1 & 1 & 1 \\ 1 & 3 & -1 & 1 & 1 & 1 & 1 \\ 1 & -1 & 3 & 1 & 1 & 1 & 1 \\ -1 & -1 & -1 & 3 & -1 & -1 & -1 \\ 1 & 1 & -1 & -1 & 3 & 1 & 1\\ 1 & 1 & -1 & -1 & 1 & 3 & 1 \\ 1 & 1 & -1 & -1 & 1 & 1 & 3  \end{pmatrix}$ & 4 & 3 & $ E_7^* $\\
stf42 & $\begin{pmatrix} 2 & 0 & 1 & 1 & 1 & 1 & 1 \\ 0 & 2 & 1 & 1 & 1 & 1 & 1 \\ 1 & 1 & 2 & 1 & 1 & 1 & 1 \\ 1 & 1 & 1 & 2 & 1 & 1 & 1 \\ 1 & 1 & 1 & 1 & 2 & 1 & 1\\ 1 & 1 & 1 & 1 & 1 & 2 & 1 \\ 1 & 1 & 1 & 1 & 1 & 1 & 2  \end{pmatrix}$ & 13 & 12 & $ P_7^* $\\
stf63 & $\begin{pmatrix} 2 & 0 & 0 & 1 & 1 & 1 & 1 \\ 0 & 2 & 1 & 1 & 1 & 1 & 1 \\ 0 & 1 & 2 & 1 & 1 & 1 & 1 \\ 1 & 1 & 1 & 2 & 1 & 1 & 1 \\ 1 & 1 & 1 & 1 & 2 & 1 & 1\\ 1 & 1 & 1 & 1 & 1 & 2 & 1 \\ 1 & 1 & 1 & 1 & 1 & 1 & 2  \end{pmatrix}$ & 4 & 3 & $ E_7 $\\
\end{tabular}
\end{center}

%
\end{document}